\documentclass[a4paper,11pt]{amsart}

\scrollmode
\usepackage{amsmath,amssymb,amsfonts,graphicx,color,fancyhdr,amsthm}
\usepackage[colorlinks=true, pdfstartview=FitV, linkcolor=blue, citecolor=blue, urlcolor=blue]{hyperref}

\newtheorem{thm}{Theorem}[section]
\newtheorem{prop}[thm]{Proposition}
\newtheorem{lem}[thm]{Lemma}
\newtheorem{df}[thm]{Definition}

\newtheorem{rem}[thm]{Remark}
\newtheorem{cor}[thm]{Corollary}
\newtheorem{ass}[thm]{Assumption}

\def\be#1 {\begin{equation} \label{#1}}
\newcommand{\ee}{\end{equation}}

\def\sqw{\hbox{\rlap{\leavevmode\raise.3ex\hbox{$\sqcap$}}$%
\sqcup$}}
\def\findem{\ifmmode\sqw\else{\ifhmode\unskip\fi\nobreak\hfil
\penalty50\hskip1em\null\nobreak\hfil\sqw
\parfillskip=0pt\finalhyphendemerits=0\endgraf}\fi}

\newcommand{\mb}{\medskip\noindent}

\newcommand{\R}{\mathbb R}

\newcommand{\N}{\mathbb N}

\newcommand{\C}{\mathbb C}

\newcommand{\A}{{\mathcal A}}
\newcommand{\K}{\kappa}
\newcommand{\tphi}{{\widetilde{\phi}}}

\newcommand{\liea}{\mathfrak{g}}
\newcommand{\cg}{\mathbb{G}}

\def\Xint#1{\mathchoice
   {\XXint\displaystyle\textstyle{#1}}%
   {\XXint\textstyle\scriptstyle{#1}}%
   {\XXint\scriptstyle\scriptscriptstyle{#1}}%
   {\XXint\scriptscriptstyle\scriptscriptstyle{#1}}%
   \!\int}
\def\XXint#1#2#3{{\setbox0=\hbox{$#1{#2#3}{\int}$}
     \vcenter{\hbox{$#2#3$}}\kern-.5\wd0}}

\def\aver#1{\Xint-_{#1}}

\newcommand{\M}{{\mathcal M}}

\begin{document}

\author{Fr\'ed\'eric Bernicot}
\address{Fr\'ed\'eric Bernicot - CNRS - Universit\'e Lille 1 \\ Laboratoire de math\'ematiques Paul Painlev\'e \\ 59655 Villeneuve d'Ascq Cedex, France}
\curraddr{}
\email{frederic.bernicot@math.univ-lille1.fr}

\author{Yannick Sire}
\address{Yannick Sire - LATP-UMR6632-Universit\'e Paul C\'ezanne\\13397 Marseille, France}
\email{sire@cmi.univ-mrs.fr}

\date{\today}

\title[Para-differential calculus on a manifold]{Propagation of low regularity for solutions of nonlinear PDEs on a Riemannian manifold with a sub-Laplacian structure}

\subjclass[2000]{...}

\maketitle

\begin{abstract}
Following \cite{B2}, we introduce a notion of para-products associated to a semi-group. We do not use Fourier transform arguments and the background manifold is doubling, endowed with a sub-laplacian structure. Our main result is a paralinearization theorem in a non-euclidean framework, with an application to the propagation of regularity for some nonlinear PDEs.  
\end{abstract}

\tableofcontents
The theory of paradifferential calculus was introduced by Bony in \cite{bony} and developed by many others, particularly Meyer in \cite{meyer}. This tool that arises is quite powerful in nonlinear analysis. The key idea relies on 
Meyer's formula for a nonlinarity $F(f)$ as $M (x, D)f + R$ where $F$ is smooth in its argument(s), $f$ belongs to a H\"older or Sobolev space, $M (x, D)$ is a pseudodifferential operator (depending on $f$) of type $(1, 1)$ and $R$ is more regular than $f$ and $F(f)$. This operation is called the ``paralinearization".

Such an approach has given many important results (or improvements of existing results): Moser estimates, elliptic regularity estimates, Kato-Ponce inequalities, ...  and is the basis of microlocal analysis.

The notion of paradifferential operators is built on appropriate functional calculus and symbolic representation, available on the Euclidean space. The Fourier transform is crucial by this point of view to study and define the symbolic classes. That is why this approach cannot be extended to Riemannian manifolds.

However, for the last years, numerous works deal with nonlinear PDEs on manifolds. So it seems important to try to extend this tool of ``paralinearization'' in a non-Euclidean situation.

First, on specific situations, namely on a Carnot group it is possible to define a suitable Fourier transform, involving irreductible representations. In this context, we can also define the notion of symbols and so of pseudo-differential calculus (see the survey \cite{BFG} of Bahouri, Fermanian-Kammerer and Gallagher for Heisenberg groups and \cite{GS} of Gallagher and Sire for more general Carnot groups). Excepted this particular setting, no Fourier transform are known.

Following this observation, the aim of this current work is to define another suitable notion of paralinearization on a manifold, without requiring use of Fourier transform. Since (nonlinear) PDEs on a manifold usually requires vector fields, we work on a manifold having a sub-Riemannian structure.
To define a suitable paralinearization, we use paraproducts defined via the heat semigroup (introduced by Bernicot in \cite{B2}, independently by Frey in \cite{Phd, Frey} and already used by Badr, Bernicot and Russ in \cite{BBR} to get Leibniz type estimates and algebra properties for Sobolev spaces) and look for a paralinearization result. 
However, a new phenomenom appears (due to the lack of flexibility of the method), the classical paralinearization result holds only for low regularity. More precisely, we prove the following (see Section \ref{sec:sobo} for Sobolev spaces and Section \ref{sec:para} for the definition of the paraproduct $\Pi$): let consider Sobolev spaces associated to a Sub-Laplacian operator on a Riemannian manifold then (under usual assumptions), we have

\begin{thm} Consider $p\in(1,\infty)$, $s\in(d/p,1)$ and $f\in W^{s+\epsilon,p}$ for some $\epsilon>0$ (as small as we want). Then for every smooth function $F\in C^\infty(\R)$ with $F(0)=0$,
\begin{equation} F(f) = \Pi_{F'(f)}(f) + w  \end{equation}
with $w\in W^{2s-d/p,p}$.
\end{thm}

As in the Euclidean situation, we are able to obtain some applications concerning propagation of the regularity for solutions of nonlinear PDEs.

With respect to the well-known paralinearization results, the first point is that we have only a gain of regularity at order $s-d/p-\epsilon$ and the main difference is that this result is only proved for $s<1$. This condition can be viewed as very strong but we will explain in Remark \ref{rem:imp} how to allow larger regularity $s>1$ with adding extra terms in the paraproducts (involving the higher order derivatives of the nonlinearity $F$). This limitation $s<1$ already appeared in \cite{CRT} and \cite{BBR}: the use of higher order Sobolev spaces require to understand higher order Riesz transforms and some cancellations properties of the iterated laplacians $\Delta^k$ ... which seems to be very difficult.

\bigskip

\section{Preliminaries : Riemannian structure with a sub-Laplacian operator} \label{sec:pre}

In this section, we aim to describe the framework and the required assumptions, we will use after. Let us precise the main hypothesis about the manifold $M$ and the operator $L$.

\subsection{Structure of doubling Riemannian manifold}

In all this paper, $M$ denotes a complete Riemannian manifold. We write $\mu$ for the Riemannian measure on $M$, $\nabla$ for the
Riemannian gradient, $|\cdot|$ for the length on the tangent space (forgetting the subscript $x$ for simplicity) and
$\|\cdot\|_{L^p}$ for the norm on $ L^p:=L^{p}(M,\mu)$, $1 \leq p\leq +\infty.$  We denote by $B(x, r)$ the open ball of
center $x\in M $ and radius $r>0$. 

\subsubsection{The doubling property}

\begin{df}[Doubling property] Let $M$ be a Riemannian manifold. One says that $M$ satisfies the doubling property $(D)$ if there exists a constant $C>0$, such that for all $x\in M,\, r>0 $ we have
\begin{equation*}\tag{$D$}
\mu(B(x,2r))\leq C \mu(B(x,r)).
\end{equation*}
\end{df}

\begin{lem} Let $M$ be a Riemannian manifold satisfying $(D)$ and let $d:=log_{2}C$. Then for all $x,\,y\in M$ and $\theta\geq 1$
\begin{equation}\label{eq:d}
\mu(B(x,\theta R))\leq C\theta^{d}\mu(B(x,R))
\end{equation}
There also exists $c$ and $N\geq 0$, so that for all $x,y\in M$ and $r>0$
\be{eq:N} \mu(B(y,r)) \leq c\left(1+\frac{d(x,y)}{r} \right)^N \mu(B(x,r)). \ee
\end{lem} 
\noindent For example, if $M$ is the Euclidean space $M=\R^d$ then $N=0$ and $c=1$. \\
Observe that if $M$ satisfies $(D)$ then
$$ \textrm{diam}(M)<\infty\Leftrightarrow\,\mu(M)<\infty\,\textrm{ (see \cite{ambrosio1})}. $$
Therefore if $M$ is a complete Riemannian manifold satisfying $(D)$ then $\mu(M)=\infty$.

\begin{thm}[Maximal theorem]\label{MIT} (\cite{coifman2})
Let $M$ be a Riemannian manifold satisfying $(D)$. Denote by $\M$ the uncentered Hardy-Littlewood maximal function
over open balls of $M$ defined by
 $$ \M f(x):=\underset{\genfrac{}{}{0pt}{}{Q \ \textrm{ball}}{x\in Q}} {\sup} \ |f|_{Q} $$
 where $\displaystyle f_{E}:=\aver{E} f d\mu:=\frac{1}{\mu(E)}\int_{E}f d\mu.$
Then for every  $p\in(1,\infty]$, $\M$ is $L^p$ bounded and moreover of weak type $(1,1)$.
\\
Consequently for $s\in(0,\infty)$, the operator $\M_{s}$ defined by
$$ \M_{s}f(x):=\left[\M(|f|^s)(x) \right]^{1/s} $$
is of weak type $(s,s)$ and $L^p$ bounded for all $p\in(s,\infty]$.
\end{thm}

Doubling property allows us to control the growth of ball-volumes. However, it can be interesting to have a lower-bound too. So we will make this following assumption:

\begin{ass} \label{ass:minoration} We assume that there exists a constant $c>0$ such that for all $x\in M$
\begin{equation} \mu(B(x,1)) \geq c. \label{eq:minoration} \end{equation}
Due to the homogeneous type of the manifold $M$, this is equivalent to a below control of the volume $(MV_d)$
\begin{equation} \tag{$MV_d$} \mu(B(x,r)) \gtrsim r^d \end{equation}
for all $0<r\leq 1$.
\end{ass}

\subsubsection{Poincar\'e inequality}
\begin{df}[Poincar\'{e} inequality on $M$] \label{classP} We say that a complete Riemannian manifold $M$ admits \textbf{a Poincar\'{e} inequality $(P_{q})$} for some $q\in[1,\infty)$ if there exists a constant $C>0$ such that, for every function $f\in W^{1,q}_{loc}(M)$ (the set of compactly supported Lipschitz functions on $M$) and every ball $Q$ of $M$ of radius $r>0$, we have
\begin{equation*}\tag{$P_{q}$}
\left(\aver{Q}|f-f_{Q}|^{q} d\mu\right)^{1/q} \leq C r \left(\aver{Q}|\nabla f|^{q}d\mu\right)^{1/q}.
\end{equation*}
\end{df}
\begin{rem} By density of $C_{0}^{\infty}(M)$ in $W^{1,q}_{loc}(M)$, we can replace $W^{1,q}_{loc}(M)$ by $C_{0}^{\infty}(M)$.
\end{rem}
Let us recall some known facts about Poincar\'{e} inequalities with varying $q$.
 \\
It is known that $(P_{q})$ implies $(P_{p})$ when $p\geq q$ (see \cite{hajlasz4}). Thus, if the set of $q$ such that
$(P_{q})$ holds is not empty, then it is an interval unbounded on the right. A recent result of S. Keith and X. Zhong
(see \cite{KZ}) asserts that this interval is open in $[1,+\infty[$~:

\begin{thm}\label{kz} Let $(X,d,\mu)$ be a complete metric-measure space with $\mu$ doubling
and admitting a Poincar\'{e} inequality $(P_{q})$, for  some $1< q<\infty$.
Then there exists $\epsilon >0$ such that $(X,d,\mu)$ admits
$(P_{p})$ for every $p>q-\epsilon$.
\end{thm}

\begin{ass} \label{ass:poincare} We assume that the considered manifold satisfies a Poincar\'e inequality $(P_1)$. Indeed we could just assume a Poincar\'e inequality $(P_\sigma)$ for some $\sigma<2$ and all of our results will remain true for Lebesgue exponents bigger than $\sigma$
\end{ass}

\subsection{Framework of Semigroup} \label{subsec:semigroup}

Let us recall the framework of \cite{DY1, DY}. \\
Let $\omega \in[0,\pi/2)$. We define the closed sector in the complex plane ${\mathbb C}$ by
$$ S_\omega:= \{z\in\C,\ |\textrm{arg}(z)|\leq \omega\} \cup\{0\}$$
and denote the interior of $S_\omega$ by $S_\omega^0$.
We set $H_\infty(S^0_\omega)$ for the set of bounded holomorphic functions $b$ on $S_\omega^0$, equipped with the norm
$$ \|b\|_{H_\infty(S_\omega^0)} := \|b\|_{L^\infty(S_\omega^0)}.$$
Then consider a linear operator $L$. It is said of type $\omega$ if its spectrum $\sigma(L)\subset S_\omega$ and for each $\nu>\omega$, there exists a constant $c_\nu$ such that
$$ \left\|(L-\lambda)^{-1} \right\|_{L^2\to L^2} \leq c_\nu |\lambda|^{-1}$$
for all $\lambda\notin S_\nu$.

\mb We refer the reader to \cite{DY1} and \cite{Mc} for more details concerning holomorphic calculus of such operators. In particular, it is well-known that $L$ generates a holomorphic semigroup $(\A_z:=e^{-zL})_{z\in S_{\pi/2-\omega}}$. Let us detail now some assumptions, we made on the semigroup.

\mb Assume the following conditions: there exists a positive real $m>0$ and $\delta>1$ with 
\begin{ass} \label{ass:semigroup}
\begin{itemize}
 \item For every $z\in S_{\pi/2-\omega}$, the linear operator $\A_z:=e^{-zL}$ is given by a kernel $a_z$ satisfying
\be{eq:pointwise} \left|a_{z}(x,y)\right|\lesssim \frac{1}{\mu(B(x,|z|^{1/2}))} \left(1+\frac{d(x,y)}{|z|^{1/2}}\right)^{-d-2N-\delta} \ee
where $d$ is the homogeneous dimension of the space (see (\ref{eq:d})) and $N$ is the other dimension parameter (see (\ref{eq:N})); $N\geq0$ could be equal to $0$.
 \item The operator $L$ has a bounded $H_\infty$-calculus on $L^2$. That is, there exists $c_\nu$ such that for $b\in H_\infty(S^0_\nu)$, we can define $b(L)$ as a $L^2$-bounded linear operator and
\be{eq:holocal} \|b(L)\|_{L^2\to L^2} \leq c_\nu \|b\|_\infty. \ee 
\item The Riesz transform ${\mathcal R}:=\nabla L^{-1/2}$ is bounded on $L^p$ for every $p\in(1,\infty)$.
\end{itemize}
\end{ass}

\begin{rem} \label{rem:holo} The assumed bounded $H_\infty$-calculus on $L^2$ allows us to deduce some extra properties (see \cite{DY} and \cite{Mc})~:
\begin{itemize}
 \item Due to the Cauchy formula for complex differentiation, pointwise estimate (\ref{eq:pointwise}) still holds for the kernel of $(tL)^k e^{-tL}$ with $t>0$.
 \item For any holomorphic function $\psi \in H(S_\nu^0)$ such that for some $s>0$, $ |\psi(z)|\lesssim \frac{|z|^s}{1+|z|^{2s}},$
the quadratic functional 
$$ f \rightarrow \left( \int_0^\infty \left|\psi(tL) f \right|^2 \frac{dt}{t} \right)^{1/2}$$
is $L^2$-bounded.
\end{itemize}
\end{rem}

\begin{rem} Concerning a square estimate on the gradient of the semigroup, it follows that for every integer $k\geq 0$ the square functional
\be{ass:square} f \rightarrow \left( \int_0^\infty \left|t^{1/2}\nabla (tL)^k e^{-tL}(f) \right|^2 \frac{dt}{t} \right)^{1/2} \ee
is bounded on $L^2$. Indeed, this is just a direct consequence of the boundedness of the previous square functions, by making appear the Riesz transform and uses its $L^2$-boundedness.
\end{rem}

\begin{rem} \label{rem:square} We claim that Assumption (\ref{ass:square}) is satisfied under the $L^2$-boundedness of the Riesz transform ${\mathcal R}:= \nabla L^{-1/2}$. \\
Indeed if ${\mathcal R}$ is $L^2$-bounded, then it admits $L^2$-valued estimates, which yield
$$ \left( \int_0^\infty \left| {\mathcal R}  (tL)^{k+1/2} e^{-tL}(f) \right|^2 \frac{dtd\mu }{t} \right)^{1/2} \leq \|{\mathcal R}\|_{L^2\to L^2} \left( \int_0^\infty \left| (tL)^{k+1/2} e^{-tL}(f) \right|^2 \frac{dtd\mu }{t} \right)^{1/2}.$$
This gives the desired result 
$$\left( \int_0^\infty \left| t^{1/2} \nabla (tL)^k e^{-tL}(f) \right|^2 \frac{dtd\mu }{t} \right)^{1/2} \lesssim \|f\|_{L^2},$$
thanks to Remark \ref{rem:holo}.
\end{rem}

About different square functions, we have the following proposition:

\begin{prop} \label{prop:Lp} Under these assumptions, we know that the square functionals in Remark \ref{rem:holo} or in (\ref{ass:square}) are $L^p$-bounded for every $p\in(1,\infty)$.
\end{prop}

\begin{proof} Let $T$ be one of the square functions in Remark \ref{rem:holo}. We also already know that it is $L^2$ bounded, by holomorphic functional calculus.
Then consider the ``oscillation operator'' at the scale $t$: 
$$B_t:=1-{\mathcal A}_t=1-e^{-tL}=-\int_0^tL e^{-sL} ds. $$
Then, by using differentiation of the semigroup, it is classical that $T B_t$ satisfies $L^2-L^2$ off-diagonal decay at the scale $t^{1/m}$, since the semigroup $e^{-tL}$ is bounded by Hardy Littlewood maximal function ${\mathcal M}_{s_-}$. So we can apply interpolation theory (see \cite{BZ} for a very general exposition of such arguments) and prove that $T$ is bounded on $L^p$ for every $p\in(s_-,2]$ (and then for $p\in[2,\infty)$ by applying a similar reasoning with the dual operators). \\
Then consider a square function $U$ of type (\ref{ass:square}). Then by using the Riesz transform, it yields
$$ U(f) =\left( \int_0^\infty \left| {\mathcal R} \psi(tL) f \right|^2 \frac{dt}{t} \right)^{1/2}$$
with $\psi(z)=z^{1/m}\phi(z)$. Since ${\mathcal R}$ is supposed to be $L^p$-bounded, it verifies $\ell^2$-valued inequalities and so the $L^p$-boundedness of $U$ is reduced to the one of a square functional of previous type, which was before proved.
\end{proof}

\subsection{Framework of sub-Laplacien}

We will only consider operators $L$ which are sub-Laplaciens and generating a semigroup $(e^{-tL})_{t>0}$ satisfying the above assumptions. Let us first explain what a sub-Laplacian means :

We assume that there exists $X=\{X_k\}_{k=1,...,\K}$ a finite family of real-valued vector fields (so $X_k$ is defined on $M$ and $X_k(x)\in TM_x$) such that
\be{L} L=-\sum_{k=1}^\K X_k^2.\ee
 We identify the ${X_k}$'s with the first order
differential operators acting on Lipschitz functions  defined on
$M$ by the formula
$$  X_kf(x)=X_k(x)\cdot \nabla f(x), $$
and we set $Xf=(X_1 f, X_2f,\cdots, X_\K f)$ and
$$  |Xf(x)|=\left(\sum_{k=1}^\K |X_k f(x)|^2\right)^{1/2}, \quad
x\in M.$$
We define also the higher-order differential operators as follows : for $I\subset \{1,...,\K\}^k$, we set
$$  X_I := \prod_{i\in I} X_{i}.$$
We assume the following:
\begin{ass} \label{ass:riesz}
For every subset $I$, the $I$th-local Riesz transform ${\mathcal R}_I:=X_I (1+L)^{-|I|/2}$ and its adjoint ${\mathcal R}_I^*:=(1+L)^{-|I|/2} X_I $ are bounded on $L^p$ for every $p\in(1,\infty)$.
\end{ass}

\begin{rem} It is easy to check that this last assumption is implied by the boundedness of each local-Riesz transform ${\mathcal R}_i$ and ${\mathcal R}^*_i$ in Sobolev spaces $W^{k,p}$ for every $p\in(1,\infty)$ and $k\in \N$. Indeed for $I=\{i_1,...,i_n\}$, we have
$$ \| {\mathcal R}_I^* f\|_{L^p} \leq \|{\mathcal R}_{i_1}^*(X_{i_2}...X_{i_n} f)\|_{W^{|I|-1,p}} \lesssim \|{\mathcal R}_{i_2,...,i_n}^* f\|_{L^p}.$$
Repeating this reasoning, we obtain that the Sobolev boundedness of the Riesz transforms and of its adjoint implies the previous Assumption.
\end{rem}

From now on, we will consider a doubling Riemannian manifold $M$ satisfying Poincar\'e inequality $(P_1)$, lower bound of the volume Assumption (\ref{ass:minoration}) and a structure of sub-Riemannian laplacian associated to a semigroup satisfying Assumption (\ref{ass:semigroup}) and with bounded Riesz transforms (Assumption (\ref{ass:riesz})).

\section{Examples of such situations}

In this section, we would like to give two examples of situations where all these assumptions are satisfied.

\subsection{Laplacian operators on Carnot-Caratheodory spaces}

Let $\Omega$ be an open connected subset of ${\mathbb R}^d$ and
$Y=\{Y_k\}_{k=1}^\kappa$ a family of real-valued, infinitely
differentiable vector fields. 

\begin{df} Let $\Omega$ and $Y$ be as above. $Y$ is said to satisfy H\"ormander's condition in
$\Omega$ if the family of commutators of vector fields in $Y$ ($Y_i$, $[Y_{i}, Y_{j}]$, ....) span ${\mathbb R}^d$ at every point of $\Omega.$
\end{df}

Suppose that $Y=\{Y_k\}_{k=1}^M$ satisfies H\"ormander's
condition in $\Omega.$  Let $C_Y$ be the family of absolutely
continuous curves $\zeta:[a,b]\to\Omega,$ $a\le b,$ such that
there exist measurable functions $c_j(t),$ $a\le t\le b,$
$j=1,\cdots, M,$ satisfying $\sum_{j=1}^{M} c_j(t)^2\le 1$ and
$\zeta'(t)=\sum_{j=1}^{M}c_j(t) Y_j(\zeta(t))$ for almost every
$t\in [a,b].$ If $x,\,y\in\Omega$ define
\[\rho(x,y)=\inf\{T>0:\text{ there exists }\zeta \in C_Y \text{ with } \zeta(0)=x \text{ and }\zeta(T)=1\}.\]
The function $\rho$ is in fact a metric in $\Omega$ called the
Carnot-Carath\'eodory metric associated to $Y.$  
This allows us to equipped the space $\Omega$ of a sub-Riemannian structure.

\subsection{Lie groups}

Let $M=G$ be a unimodular connected Lie group endowed with its Haar measure $d\mu=dx$ and assume that it has a polynomial volume growth. Recall that ``unimodular" means that $dx$ is both left-invariant and right-invariant. Denote by ${\mathcal L}$ the Lie algebra of $G$. Consider a family $X =
\{X\, ... ,X_\kappa\}$ of left-invariant vector fields on $G$ satisfying the H\"ormander condition, which
means that the Lie algebra generated by the $X_i$'s is ${\mathcal L}$. By ``left-invariant," one
means that, for any $g \in G$ and any  $f\in C^\infty_0(G)$, $X(\tau_g f)
= \tau_g(Xf)$, where $\tau_g$ is the left-translation operator.
As previously, we can build the Carnot-Carath\'eodory metric on $G$. The left
invariance of the $X_i$'s implies the left-invariance of the distance $d$. So that for every $r$, the volume of the ball $B(x,r)$ does not depend on $x\in G$ and also will be denoted $V(r)$.
It is well-known (see \cite{Gui,NSW}) that $(G,d)$ is then a space of homogeneous type. Particular case are Carnot groups, where the vector fields are given by a Jacobian basis of its Lie algebra and satisfy H\"ormander condition. In this situation, two cases may occur : either the manifold is doubling or the volume of the balls admit an exponential growth \cite{Gui}. For example, nilpotents Lie groups satisfies the doubling property (\cite{16}).

We refer the reader to \cite[Thm5.14]{Robinson} and \cite[Section 3, Appendix 1]{CRT} where properties of the heat semigroup are studied: in particular the heat semigroup $e^{-tL}$ satisfies gaussian upper-bounds and Assumption \ref{ass:semigroup} on the higher-order Riesz transforms (Assumption (\ref{ass:riesz}) is satisfied too.

\subsection{Carnot groups}

 Particular cases of nilpotents Lie groups are the Carnot groups (if it admits a stratification).
 A stratification on a Lie group $G$ (whose $\liea$ is its Lie algebra) is a collection of linear subspaces $V_1$, ..., $V_r$ of $G$ such that
$$ \liea = V_1 \oplus ... \oplus V_r$$
which satisfy $[V_1, V_i] = V_{i+1}$ for $i=1, ..., r-1$ and $[V_1, V_r] = 0$.
By $[V_1, V_i]$ we mean the subspace of $G$ generated by the elements $[X, Y]$
where $X \in V_1$ and $Y \in V_i$. Consider  $n_i$ the dimension of $V_i$, $n_1+\cdots+n_r=d$ and dilations $\{\delta_\lambda\}_{\lambda>0}$ of the form
\[\delta_\lambda(x)=(\lambda \,x^{(1)}, \lambda^2 x^{(2)},\cdots,\lambda^s x^{(r)}),\quad  x^{(i)}\in V_i.\]
The couple $\cg=(G,\delta_\lambda)$ is said to be a homogeneous Carnot group (of step $r$ and $n_1$ generators) if  $\delta_\lambda$ is an automorphism of $G$ for every $\lambda>0$ and if the first $n_1$ elements of the Jacobian basis of $\liea,$ say $Z_1,\cdots,Z_{n_1},$  satisfy
\begin{equation}\label{hormander}
\text{rank}(\text{Lie}[Z_1,\cdots,Z_{n_1}](x))=d,\qquad \text{for all }x\in G,
\end{equation}
where $\text{Lie}[Z_1,\cdots,Z_{n_1}]$ is the Lie algebra generated by the vector fields $Z_1,\cdots Z_{n_1}.$
The number $Q=\sum_{i=1}^r i\,n_i$ is called the homogeneous dimension of $\cg$. 

As for example the different Heisenberg groups, ${\mathcal H}^d$ is a Carnot group of dimension $Q = 2d + 2$. 
We refer the reader to \cite{GS} for an introduction of pseudodifferential operators in this context using a kind of Fourier transforms involving irreductible representations (and to \cite{BFG} for a complete work about pseudo-differential calculus on Heisenberg groups).

\subsection{Riemannian manifolds with a bounded geometry}

We shall say that a Riemannian manifold $M$ has a bounded geometry if 
\begin{itemize}
\item the curvature tensor and all its derivatives are bounded
\item Ricci curvature is bounded from below
\item and $M$ has a positive injectivity radius.
\end{itemize}

In such situations, we know that there exists a collection of smooth vector fields $X_1,...,X_\K$ such that 
$$ \Delta = -\sum_{i=1}^\K X_i^2.$$
Moreover Assumptions \ref{ass:semigroup} and \ref{ass:riesz} are satisfied (see \cite{CRT} and \cite{Triebel}).

\section{The scale of Sobolev spaces} \label{sec:sobo}

We use the Bessel-type Sobolev spaces, adapted to the operator $L$:
\begin{df} For $p\in(1,\infty)$ and $s\geq 0$, we set
$$ W^{s,p}=W^{s,p}_L:=\left\{f\in L^p,\  (1+L)^{s/2}(f) \in L^p \right\}.$$ 
\end{df}

First, we have this characterization:

\begin{prop} \label{prop:equivalence} For all $p\in (1,\infty)$ and $s>0$, we have the following equivalence
$$ \|f\|_{L^p} + \| L^{s/2}(f)\|_{L^p} \simeq \|(1+L)^{s/2} f \|_{L^p}.$$
\end{prop}

\begin{proof} Set $\alpha=s/2$ and write $\alpha=k+\theta$ with $k\in\N$ and $\theta\in[0,1)$. We decompose $(1+L)^{\alpha}$ with the semigroup as following
\begin{align*}
 (1+L)^{\alpha} f   & = \int_0^\infty e^{-t} e^{-tL} (1+L) t^{1-\theta} \frac{dt}{t} (1+L)^{k}(f) \\
  & = \int_0^\infty e^{-t} e^{-tL} t^{1-\theta} \frac{dt}{t} (1+L)^{k}(f) + \int_0^\infty e^{-t} e^{-tL} (tL)^{1-\theta} \frac{dt}{t} L^\theta (1+L)^{k}(f).
\end{align*}
The first integral-operator is easily bounded on $L^p$ since the semigroup $e^{-tL}$ is uniformly bounded. The second integral operator is bounded using duality:
\begin{align*}
 & \langle \int_0^\infty e^{-t} e^{-tL} (tL)^{1-\theta}( u) \frac{dt}{t},g \rangle  =  \int_0^\infty e^{-t} \langle e^{-tL/2} (tL)^{\frac{1-\theta}{2}}(u), e^{-tL^*/2} (tL^*)^{\frac{1-\theta}{2}} g \rangle \frac{dt}{t}  \\
 & \hspace{1cm} \leq \int \left(\int_0^\infty \left|e^{-tL/2} (tL)^{\frac{1-\theta}{2}}(u)\right|^2 \frac{dt}{t}\right)^{1/2} \left(\int_0^\infty \left|e^{-tL^*/2} (tL^*)^{\frac{1-\theta}{2}}(g)\right|^2 \frac{dt}{t}\right)^{1/2} d\mu.
\end{align*}
Since $(1-\alpha)/2>0$, then the two square functionals are bounded in $L^p$ and $L^{p'}$ (by Proposition \ref{prop:Lp}) and that concludes the proof of
$$ \|(1+L)^{\alpha} f \|_{L^p} \lesssim \|(1+L)^k f\|_{L^p} + \| L^{\theta}(1+L)^k(f)\|_{L^p}.$$
Then, developing $(1+L)^k$, it follows a finite sum of $\|L^z(f)\|_{L^p}$ with $z\in[0,\alpha]$.
We decompose
$$ L^z(f) = \int_0^\infty e^{-tL} (tL)^\alpha(f) t^{-z}\frac{dt}{t} = \int_0^1 e^{-tL} (tL)^\alpha(f) t^{-z}\frac{dt}{t}+ \int_1^\infty e^{-tL} (tL)^\alpha t^{-z}\frac{dt}{t}.$$
The first quantity in $L^p$ is controlled by $\|L^\alpha(f)\|_{L^p}$ and the second one by $\|f\|_{L^p}$, which concludes the proof of
$$  \|(1+L)^{s/2} f \|_{L^p} \lesssim \|f\|_{L^p} + \| L^{s/2}(f)\|_{L^p} .$$
Let us now check the reverse inequality. As previously, for $u=0$ or $u=\alpha$ we write
$$ L^u f = \int_0^\infty e^{-t(1+L)} (1+L) L^u t^{1+\alpha} \frac{dt}{t} (1+L)^\alpha f.$$
By producing similar arguments as above, the operator $\int_0^\infty e^{-t(1+L)} (1+L) L^u t^{1+\alpha} \frac{dt}{t}$ is easily bounded on $L^p$ (splitting the integral for $t\leq 1$ and $t\geq 1$) and we can also conclude to
$$ \|L^u(f)\|_{L^p} \lesssim \|(1+L)^\alpha f\|_{L^p},$$
which ends the proof.
\end{proof}

\begin{cor} For all $p\in (1,\infty)$ and $0\leq t\leq s$, we have the following inequality
$$ \|L^{t}f\|_{L^p} \lesssim \|(1+L)^{t} f \|_{L^p} \simeq \|f\|_{W^{2t,p}}.$$
 \end{cor}

Let us then describe classical Sobolev embeddings in this setting (see \cite{BBR} for a more general framework):

\begin{prop} \label{prop:Sobolev} Under Assumption \ref{ass:minoration} (lower bound on the ball-volumes), let $s\geq t\geq 0$ be fixed and take $p\leq q$ such that
$$ \frac{1}{q}-\frac{t}{d} > \frac{1}{p}-\frac{s}{d}.$$
Then, we have the continuous embedding
$$ W^{s,p} \hookrightarrow W^{t,q}.$$
\end{prop}

We refer the reader to \cite[Proposition 3.3]{BBR} for a precise proof. The proof is based on a spectral decomposition, to write the resolvant with the semigroup and then to use the off-diagonal estimates (here the pointwise estimates on the heat kernel).

\begin{cor} \label{cor:Li} Under the previous assumption, $ W^{s,p} \hookrightarrow L^\infty$ as soon as 
$$ s>\frac{d}{p}.$$
\end{cor}

We now recall a result of \cite{BBR}, where a characterization of Sobolev spaces is obtained, involving some fractional functionals. 

\begin{prop} \cite[Thm 5.2]{BBR} \label{prop:cara} Under Poincar\'e inequality $(P_1)$, for $s\in(0,1)$ we have the following characterization : a function $f\in L^p$ belongs to $W^{s,p}$ if and only if
$$ S_{s}^{\rho,loc} f(x)=\left(\int_{0}^{1} \left[\frac{1}{r^{s}} \left(\frac{1}{\mu(B(x,r))}\int_{B(x,r)}|f(y)-f(x)|^\rho d\mu(y)\right)^{1/\rho}\right]^2\frac{dr}{r}\right)^{\frac{1}{2}}
$$
belong to $L^p$, for some $\rho < \min(2,p)$.
\end{prop}

This characterization can be extended for $s>1$, using the sub-Laplacian structure. Indeed, we have this first Lemma:
\begin{lem} For every integer $k$ and $p\in(1,\infty)$, 
$$ \|f\|_{W^{k,p}} \simeq \sum_{I\subset \{1,...,\K\}^k} \|X_I(f)\|_{L^p}.$$
\end{lem}

\begin{proof} As point out in \cite{CRT}, this is consequence of Assumption \ref{ass:riesz} about the local Riesz transforms. Indeed, for $k\geq 1$ and $I$ a subset, we have assumed that the $I$-th Riesz transform ${\mathcal R}_I$ are bounded on $L^p$, which is equivalent to
$$ \|X_I(f)\|_{L^p} \lesssim \|f\|_{W^{|I|/2,p}}.$$
Moreover, making appear the adjoint of the Riesz transforms and the resolvant (which are all bounded on $L^p$) as follows
$$ (1+L)^{1/2}= (1+L)^{-1/2} (1+L) = (1+L)^{-1/2} + \sum_{i=1}^{\K} (1+L)^{-1/2} X_i^2 = (1+L)^{-1/2} + \sum_{i=1}^{\K} {\mathcal R}_i^* X_i, $$
we conclude to the reverse inequality and so we have proved the desired result for $k=1$.
We let the details for $k\geq 2$ to the reader, the reasoning is exactly the same but technically to be written (we make appear a finite sum of higher order Riesz transforms ...)
\end{proof}

We also deduce the following characterization (see Proposition 19 in \cite{CRT}):

\begin{prop} Let $s:=k+t>1$ (with $k$ an integer and $t \in(0,1)$), then
\begin{align*}
 f\in W^{s,p} &   \Longleftrightarrow f\in L^p \textrm{  and  } \forall I \subset \{1,...,\K \}^k,\  X_I(f) \in W^{t,p} \\
 & \Longleftrightarrow f\in L^p \textrm{ and }  \forall I \subset \{1,...,\K \}^k, \  S_{t}^\rho(X_I(f)) \in L^p. 
\end{align*}
\end{prop}

We also deduce the following chain rule (see Theorem 22 in \cite{CRT} for a proof by induction on $k$):

\begin{prop} \label{prop:compoSobolev}
 If $F\in C^\infty$ with $F(0)=0$ and let $s:=k+t>\frac{d}{p}$ (with $p\in(1,\infty)$, $k$ an integer and $t \in(0,1)$).
 Then 
 $$ \|F(f)\|_{W^{s,p}} \lesssim \|f\|_{W^{s,p}} + \|f\|_{W^{s,p}}^k.$$
 If $F(0)\neq 0$, we still have such inequalities with localized Sobolev spaces.
\end{prop}

We refer the reader to \cite{CRT} for a proof by induction on $k$. Here for completeness, we produce another direct proof.

\begin{proof} We use the previous characterization of the Sobolev space with $S_{t}^\rho$ for $s=k+t$. First using the differentiation rule, it comes $X_i(F(f))=X_i(f) F'(f)$, then $X_j X_i(F(f)) = X_jX_i(f) F'(f) + X_i(f) X_j(f) F''(f)$ ... By iterating the reasoning, for $I \subset \{1,...,\K \}^k$, estimating $X_I(F(f))$ in $W^{t,p}$ is reduced to estimate quantities as
$$ h:= \left[\prod_{\alpha=1}^l X_{i_\alpha}\right] (f) F^{(n)}(f)$$
where $i_\alpha \subset I$, $n\leq k$ and $\sum |i_\alpha| = |I|\leq k$.
Then for $x,y$, we have
\begin{align*}
 \left|h(x)-h(y)\right| \leq & \sum_{\beta} |X_{i_{\beta}}(f)(x) - X_{i_{\beta}}(f)(y)| \prod_{\alpha\neq \beta } \sup_{z=x,y} |X_{i_\alpha} (f)(z)|  \|F^{(n)}(f)\|_{L^\infty} \\
 & + \prod_{\alpha} \sup_{z=x,y} |X_{i_\alpha} (f)(z)| |F^{(n)}(f)(x) - F^{(n)}(f)(y)|. 
\end{align*}
By this way, since $\rho\leq p$ let us choose exponents $\rho_\alpha$, $p_\alpha$ such that 
$$ \frac{1}{\rho} = \sum_\alpha \frac{1}{\rho_\alpha}, \quad \rho_\alpha \leq p_\alpha$$
and
$$ \frac{1}{p}=\sum_\alpha \frac{1}{p_\alpha}.$$
Moreover we require that 
\begin{equation}\frac{1}{p_\alpha} - \frac{|i_\alpha|+t}{d} > \frac{1}{p} - \frac{s}{d}. \label{eq:sobo} \end{equation}
This is possible since $\sum_{\alpha} |i_\alpha| =|I|\leq s-t$ and $s>d/p$ (indeed we let the reader to check that $p_\alpha=\frac{|I|+t}{|i_\alpha|+t} p$ is a good choice). Moreover, we chose 
exponents $\overline{\rho_\alpha},\overline{\rho}$, $\overline{p_\alpha}$ and $\overline{p}$ such that 
$$ \frac{1}{\rho} = \sum_\alpha \frac{1}{\overline{\rho_\alpha}}+ \frac{1}{\overline{\rho}}, \quad \overline{\rho_\alpha} \leq \overline{p_\alpha}$$
and $\overline{\rho} \leq \overline{\rho}\leq \overline{p}$ with
$$ \frac{1}{p}=\sum_\alpha \frac{1}{\overline{p_\alpha}}+\frac{1}{\overline{p}}.$$
As previously, we require (\ref{eq:sobo}) with $\overline{p_\alpha}$ instead of $p_\alpha$ and  
\begin{equation}\frac{1}{\overline{p}} > \frac{1}{p} - \frac{s}{d}. \label{eq:sobo2} \end{equation}
Such exponents can be chosen by perturbing the previous construction with a small parameter since $s>d/p$.
By this way (with H\"older inequality), we deduce that 
\begin{align*} 
S_t^\rho(h) & \lesssim \sum_{\beta} S_t^{\rho_\beta}(X_{i_{\beta}}(f)) \prod_{\alpha\neq \beta } \M_{\rho_\alpha} [X_{i_\alpha} (f)] \|F^{(n)}(f)\|_{L^\infty} \\
 & \ \ + \prod_{\alpha} \M_{\overline{\rho_\alpha}} [X_{i_\alpha} (f)] S_t^{\overline \rho}(F^{(n)}(f) .
\end{align*}
Since $F$ is supposed to be bounded in $C^\infty$, then $F^{(n)}$ is Lipschitz and so, we finally obtain
\begin{align*} S_t^\rho(h) & \lesssim \sum_{\beta} S_t^{\rho_\beta}(X_{i_{\beta}}(f)) \prod_{\alpha\neq \beta } \M_{\rho_\alpha} [X_{i_\alpha} (f)] \|F^{(n)}(f)\|_{L^\infty} \\
 & + \prod_{\alpha} \M_{\overline{\rho_\alpha}} [X_{i_\alpha} (f)] S_t^{\overline{\rho}}(f).
\end{align*}
Then applying H\"older inequality, we get
\begin{align*}
\| S_t^\rho(h)\|_{L^p} & \lesssim \sum_{\beta} \left\| S_t^{\rho_\beta}(X_{i_{\beta}}(f))\right\|_{L^{p_\beta}}  \prod_{\alpha\neq \beta } \left\| \M_{\rho_\alpha} [X_{i_\alpha} (f)]\right\|_{L^{p_\alpha}} \|F^{(n)}(f)\|_{L^\infty} \\
& \ \ + \prod_{\alpha} \left\| \M_{\overline{\rho_\alpha}} [X_{i_\alpha} (f)]\right\|_{L^{\overline{p_\alpha}}} \left\|S_t^{\overline {\rho}}(f)\right\|_{L^{\overline{p}}}.
\end{align*}
Since (\ref{eq:sobo}) with Sobolev embeddings (Proposition \ref{prop:Sobolev}), we have
$$ \left\| \M_{\rho_\alpha} [X_{i_\alpha} (f)]\right\|_{L^{p_\alpha}} \lesssim \|f\|_{W^{|i_\alpha|,p_\alpha}} \lesssim \|f\|_{W^{s,p}}$$
and
$$  \left\| S_t^{\rho_\beta}(X_{i_{\beta}}(f))\right\|_{L^{p_\beta}} \lesssim \|f\|_{W^{|i_\beta|+t,p_\beta}} \lesssim \|f\|_{W^{s,p}}.$$
So with (\ref{eq:sobo}) and (\ref{eq:sobo2}), we finally obtain
$$ \|S_t^\rho(h)\|_{L^p} \lesssim \|f\|_{W^{s,p}}+\|f\|_{W^{s,p}}^k,$$
where we used $s>d/p$ and the Sobolev embedding $W^{s,p} \subset L^\infty$ with the smoothness of $F$ to control $\|F^{(n)}(f)\|_{L^\infty}$.

Since $F(0)=0$ and $F$ is Lipschitz, we also deduce that $F(f)$ belongs to $L^p$, which allows us to get the expected result
$$ \|F(f)\|_{W^{s,p}} \lesssim \|F(f)\|_{L^p} + \sum_{I} \| S_t^\rho(F(f))\|_{L^p} \lesssim \|f\|_{W^{s,p}} + \|f\|_{W^{s,p}}^k.$$
\end{proof}

\begin{rem} \label{rem:compoSobolev}
 If $F\in C^\infty$ with $F(0)=0$ and $s>d/p$ then to obtain 
 $$ \|F(f)\|_{W^{s,p}} \lesssim \|f\|_{W^{s,p}}+ \|f\|_{W^{s,p}}^k,$$
 it is sufficient to assume that $F$ is locally bounded in $C^\infty$ and then the implicit constant will depend on $\|f\|_{L^\infty}$.
 Indeed, using Sobolev embedding, we know that as soon as $s>d/p$, $W^{s,p}$ is continuously embedded in $L^\infty$.
\end{rem}

\section{Paraproducts associated to a semigroup} \label{sec:para}

Our aim is to describe a kind of ``paralinearization'' results. In the Euclidean case, this is performed by using paraproducts (defined with the help of Fourier transform). Here, we cannot use such powerful tools, so we require other kind of paraproducts, defined in terms of semigroup. These ones were introduced by the first author in \cite{B2}, already used in \cite{BBR} and more recently was extended in \cite{Phd, Frey}. Let us recall these definitions.

 \subsection{Definitions and spectral decomposition} \label{subsec:def}
 
We consider a sub-Laplacian operator $L$ satisfying the assumptions of the previous sections. 
We write for convenience $c_0$ for a suitably chosen constant, $\psi(x)=c_0 x^N e^{-x}(1-e^{-x})$ and so
$$ \psi_t(L):=c_0(tL)^N e^{-tL}(1-e^{-tL}),$$
with a large enough integer $N>d/2$.
Let $\phi$ be the function
$$ \phi(x):=-c_0\int_x^\infty y^N e^{-y}(1-e^{-y}) dy,$$
$$ \tphi(x):=-c_0\int_x^\infty y^{N-1} e^{-y}(1-e^{-y}) dy,$$
and set $ \phi_t(L):=\phi(tL)$.
Then we get a ``spectral'' decomposition of the identity as follows (choosing the appropriate constant $c_0$), we have
$$ f  = - \int_0^\infty \phi'(tL) f \frac{dt}{t}.$$
So for two smooth functions, we have
$$ fg := -\int_{s,u,v>0} \phi'(sL) \left[\phi'(uL)f \, \phi'(vL)g \right] \frac{dsdudv}{suv}.$$
Since $\phi'(x)=\psi(x):=c_0x^Ne^{-x}(1-e^{-x})$ and $x\tphi '(x)=\phi'(x)$, it comes that (by integrating according to $t:=\min\{s,u,v\}$)
\begin{equation}
\begin{split}
 fg := & - \int_{0}^\infty \psi(tL) \left[\tphi(tL)f \, \tphi(tL)g \right] \frac{dt}{t} -  \int_{0}^\infty \tphi(tL) \left[\psi(tL)f \, \tphi(tL)g \right] \frac{dt}{t}  \\
& -\int_{0}^\infty \tphi(tL) \left[\tphi(tL)f \, \psi(tL)g \right] \frac{dt}{t}.
\end{split}
 \label{eq:decomp}
\end{equation}
Let us now focus on the first term in (\ref{eq:decomp}) :
$$ I(f,g) = \int_{0}^\infty \psi(tL) \left[\tphi(tL)f \, \tphi(tL)g \right] \frac{dt}{t}.$$
Since $N>> 1$, let us write $\psi(z)=z\tilde{\psi}$ with $\tilde{\psi}$ (still vanishing at $0$ and at infinity).
Then using the structure of the sub-Laplacian $L$, the following algebra rule holds 
$$L(fg)=L(f) g + f L(g) + \langle X f\cdot X g \rangle,$$
where $X$ is the collection of vector fields $Xf:=(X_1f,...,X_\K f)$. 
Hence, we get
\begin{align*}
 I(f,g) & =  \int_{0}^\infty \tilde{\psi}(tL) (tL)\left[\tphi(tL)f \, \tphi(tL)g \right] \frac{dt}{t}\\
 & =  \int_{0}^\infty \tilde{\psi}(tL) \left[tL\tphi(tL)f \, \tphi(tL)g \right] \frac{dt}{t} \\
 & \  + \int_{0}^\infty \tilde{\psi}(tL) \left[\tphi(tL)f \, tL\tphi(tL)g \right] \frac{dt}{t} \\
 & \  + \int_{0}^\infty \tilde{\psi}(tL) t \langle X \tphi(tL)f \cdot X \tphi(tL)g \rangle \frac{dt}{t}. 
\end{align*}
Combining with (\ref{eq:decomp}), we define the paraproduct as follows~:

\begin{df} \label{def:paraproduit} With the previous notations, we define the paraproduct of $f$ by $g$, by
\begin{align*}
 \Pi_g(f) := & - \int_{0}^\infty \tilde{\psi}(tL) \left[tL \tphi(tL)f \, \tphi(tL)g \right] \frac{dt}{t} \\
  & - \int_{0}^\infty \tphi(tL) \left[\psi(tL)f \, \tphi(tL)g \right] \frac{dt}{t}.  
\end{align*}
\end{df}

\begin{rem} \label{rem:imp2} We first want to point out the difference with the initial definition in \cite{B2}. There, general semigroup was considered and the previous operation on the term $I$ can be performed by making appear the ``carr\'e du champ'' introduced by Bakry and \'Emery (see \cite{BE} for details)
$$ \Gamma(f,g) := L(fg)-L(f)g - fL(g)$$
instead of the vector field $X$.
However in \cite{B2}, the paraproducts was only defined by the second term. This new definition comes from the following observation: considering the quantity $I(f,g)$ and distributing the Laplacian as we have done (or make appearing the ``carr\'e du champ''), it comes three terms. The term $\tilde{\psi}(tL) \left[tL \tphi(tL)f \, \tphi(tL)g \right]$ has the same regularity properties as $\tphi(tL) \left[\psi(tL)f \, \tphi(tL)g \right]$ (in the sense that $tL \tphi(tL)$ can be considered as $\psi(tL)$). This also legitimate to add this extra term in the definition of the paraproducts.

By this way, as we will see in the next properties and in Remark \ref{rem:imp}, this new paraproduct is the ``maximal'' (in a certain sense) part of the product $fg$, where the regularity is given by the regularity of $f$.
\end{rem}

It naturally comes the following decomposition~:

\begin{cor} \label{cor:decomp} Let $f,g$ be two smooth functions, then we have
$$ fg = \Pi_g(f) + \Pi_f(g) + \texttt{Rest}(f,g)$$
where the ``rest'' is given by
$$ \texttt{Rest}(f,g):=- \int_{0}^\infty \tilde{\psi}(tL) \langle t^{1/2}X\tphi(tL)f , t^{1/2}X\tphi(tL)g\rangle \frac{dt}{t}. $$
 \end{cor}

\subsection{Boundedness of paraproducts in Sobolev and Lebesgue spaces}

Concerning estimates on these paraproducts in Lebesgue spaces, we refer to \cite{B2}:

\begin{thm}[Boundedness in Lebesgue spaces] For $p,q\in(1,\infty]$ with $0<\frac{1}{r}:=\frac{1}{p}+\frac{1}{q}$ then 
$$ (f,g)\rightarrow \Pi_g(f)$$ is bounded from $L^p \times L^q$ into $L^r$.
 \end{thm}

Let us now describe boundedness in the scale of Sobolev spaces.

\begin{thm}[Boundedness in Sobolev spaces] \label{thm:sobolev} For $p,q,r\in(1,\infty)$ with $\frac{1}{r}:=\frac{1}{p}+\frac{1}{q}$  and $s\in(0;2N-4)$ then 
$$ (f,g)\rightarrow \Pi_g(f)$$ is bounded from $W^{s,p} \times L^q$ into $W^{s,r}$.
\end{thm}

\begin{proof} It is sufficient to prove the following homogeneous estimates : for every $\beta\in[0,N-2)$
$$ \| L^{\beta}\Pi_g(f)\|_{L^{r}} \lesssim \|L^\beta(f)\|_{L^p} \|g\|_{L^q}.$$
For $\beta=0$, this is the previous theorem so it remains us to check it for $\beta \in (0,N-2)$.
We recall that 
\begin{align*}
 \Pi_g(f)  = & - \int_{0}^\infty \tilde{\psi}(tL) \left[tL \tphi(tL)f \, \tphi(tL)g \right] \frac{dt}{t} \\
  & - \int_{0}^\infty \tphi(tL) \left[\psi(tL)f \, \tphi(tL)g \right] \frac{dt}{t},  
\end{align*}
giving rise to two quantities, $\Pi^1_g(f)$ and $\Pi^2_g(f)$. Indeed, applying $L^\beta$ to the paraproduct $\Pi^2_g(f)$, it yields
\begin{align*} 
L^\beta \Pi^2_g(f) & =  \int_{0}^\infty L^\beta \tphi(tL) \left[\psi(tL)f \, \tphi(tL)g \right] \frac{dt}{t} \\
 & =  \int_{0}^\infty \overline{\psi}(tL) \left[t^{-\beta} \psi(tL)f \, \tphi(tL)g \right] \frac{dt}{t} \\
 & = \int_{0}^\infty \overline{\psi}(tL) \left[ \widetilde{\psi}(tL) L^\beta f \, \tphi(tL)g \right] \frac{dt}{t},
\end{align*}
where we set $\overline{\psi}(z)=z^\beta \phi(z)$ and $\widetilde{\psi}(z)=z^{-\beta} \psi(z)$. So if the integer $N$ in $\phi$ and $\psi$ is taken sufficiently large, then $\overline{\psi}$ and $\widetilde{\psi}$ are still holomorphic functions with vanishing properties at $0$ and at infinity. As a consequence, we get
$$ L^\beta \Pi_g(f)  = \overline{\Pi}_g(L^\beta f)$$
with the new paraproduct $\overline{\Pi}$ built with $\overline{\psi}$ and $\widetilde{\psi}$. We also apply the classical reasoning aiming to estimate this paraproduct.
By duality, for any smooth function $h\in L^{r'}$ we have
\begin{align*}
  \langle L^\beta \Pi^2_g(f), h\rangle   & = \int \int_{0}^\infty \overline{\psi}(tL^*)h  \, \widetilde{\psi}(tL)(L^\beta f) \, \tphi(tL)g \frac{dt}{t} d\mu \\
  & \leq \int \left(\int_0^\infty | \overline{\psi}(tL^*)h|^2 \frac{dt}{t} \right)^ {1/2} \left(\int_0^\infty |\widetilde{\psi}(tL)(L^\beta f)|^2 \frac{dt}{t} \right)^{1/2} \sup_t |\tphi(tL)g| d\mu. 
\end{align*}
From the pointwise decay on the semigroup (\ref{eq:pointwise}), we know that
$$ \sup_t |\tphi(tL)g(x)| \leq \M(g)(x)$$ 
and so by H\"older inequality
\begin{align*}
  \left| \langle L^\beta \Pi^2_g(f), h\rangle \right| \lesssim \left\| \left(\int_0^\infty |\overline{\psi}(tL^*)h|^2 \frac{dt}{t} \right)^ {1/2} \right\|_{L^{r'}} \left\|\left(\int_0^\infty |\widetilde{\psi}(tL)(L^\beta f)|^2 \frac{dt}{t} \right)^{1/2}\right\|_{L^p}  \left\| {\M} g\right\|_{L^q}. 
\end{align*}
Since $\overline{\psi}$ and $\widetilde{\psi}$ are holomorphic functions vanish at $0$ and having fast decays at infinity, we know from (\ref{prop:Lp}) that the two square functions are bounded on Lebesgue spaces. We also conclude the proof by duality, since it comes
\begin{align*}
  \left| \langle L^\beta \Pi^2_g(f), h\rangle \right| \lesssim \left\| h \right\|_{L^{r'}} \left\|L^\beta f \right\|_{L^p}  \left\| g\right\|_{L^q}. 
\end{align*}
We let the reader to check that the same arguments still holds for the first part $\Pi^1_g(f)$ and so the proof is also finished.
\end{proof}

\section{Linearization Theorem}

\begin{thm} \label{thm:linearization} Consider $s\in(d/p,1)$ and $f\in W^{s+\epsilon,p}$ for some $\epsilon>0$. Then for every smooth function $F\in C^\infty(\R)$ with $F(0)=0$,
\begin{equation} F(f) = \Pi_{F'(f)}(f) + w \label{eq:decomposit} \end{equation}
with $w\in W^{2s-d/p,p}$.
 \end{thm}

We follow the proof in \cite{cm, meyer, bony}.

\begin{proof} Let us refer the reader to the operators $\phi(tL)$ and $\psi(tL)$, defined in Subsection \ref{subsec:def}: $\psi(x)=c_0 x^N e^{-x}(1-e^{-x})$, $\phi$ is its primitive vanishing at infinity. Let us write $\widetilde{\psi}(z)=z^{-1} \psi(z)$ and $\widetilde{\phi}$ its primitive vanishing at infinity. Moreover, these functions are normalized by the suitable constant $c_0$ such that $\tphi(0)=1$.

It comes
$$ f= \lim_{t\rightarrow 0} \widetilde{\phi}(tL)(f)$$
and so we decompose 
$$ F(f) = \tphi(L) F(\widetilde{\phi}(L) f) - \int_{0}^1 \frac{d}{dt} \tphi(tL) F(\widetilde{\phi}(tL)f) dt.$$                                                                                                                                                                                                                  
 Since
 \begin{align*}
  t \frac{d}{dt} \tphi(tL) F(\widetilde{\phi}(tL)f) & = tL \tphi'(tL) F(\tphi(tL) f) + \tphi(tL) \left[ (tL \tphi'(tL) f)  F'(\widetilde{\phi}(tL) f)\right] \\
  & = \phi'(tL) F(\tphi(tL) f) + \tphi(tL) \left[ (\phi'(tL) f)  F'(\widetilde{\phi}(tL) f)\right],
 \end{align*}
we get
\begin{align*} F(f) & = \tphi(L) F(\widetilde{\phi}(L) f) - \int_0^1   \tphi'(tL) tL [F(\tphi(tL) f)] + \tphi(tL) \left[ (\phi'(tL) f)  F'(\widetilde{\phi}(tL) f)\right] \frac{dt}{t} \\
& = \tphi(L) F(\widetilde{\phi}(L) f) - \int_0^1   \tphi'(tL) \left[F''(\tphi(tL) f) |t^{1/2} X \tphi(tL) f|^2 + F'(\tphi(tL) f) tL \tphi(tL) f \right] \\
& \ + \tphi(tL) \left[ (\phi'(tL) f)  F'(\widetilde{\phi}(tL) f\right] \frac{dt}{t},
\end{align*}
where we used the differentiation rule for the composition with the vector fields $X=(X_1,...,X_\K)$.
We also set
$$ w:= I+ II+ III+ IV+V$$
with 
$$ I:= \tphi(L) F(\widetilde{\phi}(L) f),$$
$$ II:= - \int_0^1   \tphi'(tL) \left[F''(\tphi(tL) f) |t^{1/2} X \tphi(tL) f|^2 \right] \frac{dt}{t},$$
$$ III:= \int_0^1   \tphi'(tL) \left[\left(\tphi(tL) F'(f) - F'(\tphi(tL) f) \right) tL \tphi(tL) f \right] \frac{dt}{t},$$
$$ IV:= \int_0^1   \tphi(tL) \left[ (\phi'(tL) f)  \left( \tphi(tL) F'(f) - F'(\widetilde{\phi}(tL) f) \right)\right] \frac{dt}{t}, $$
and 
$$ V:= \int_1^\infty \tilde{\psi}(tL) \left[tL \tphi(tL)f \, \tphi(tL) F'(f) \right] \frac{dt}{t} + \int_{1}^\infty \tphi(tL) \left[\psi(tL)f \, \tphi(tL) F'(f) \right] \frac{dt}{t}  $$
in order that (\ref{eq:decomposit}) is satisfied. It remains us to check that each term belongs to $ W^{2s-d/p,p}$.

\mb {\bf Step 1:} Term $I$. \\
 Since $f\in W^{s+\epsilon,p}$ then $\widetilde{\phi}(L) f$ belongs to $W^{\rho,p}$ for every $\rho\geq s+\epsilon$ and so Proposition \ref{prop:compoSobolev} yields that  
$$ \left\| \tphi(L) F(\tphi(L) f) \right\|_{W^{2s-d/p,p}} \lesssim \|f\|_{W^{s+\epsilon,p}}.$$

\mb {\bf Step 2:} Term $V$. \\
We only treat the first term in $V$ (the second one can be similarly estimated).
Using duality, we have with some $g\in L^{p'}$ and for $\alpha\in \{0,2s-d/p\}$ since $\alpha\geq0$ and $t\geq 1$
\begin{align*}
\| L^{\alpha/2} V\|_{L^p} & \leq  \int \int_1^\infty \left|(tL)^{\alpha/2}\tilde{\psi}(tL^*)g\right| \left|tL \tphi(tL)f \, \tphi(tL) F'(f) \right| \frac{dtd\mu}{t} \\
& \lesssim \left\| \left( \int_1^\infty \left|tL \tphi(tL)f \, \tphi(tL) F'(f) \right|^2 \frac{dt}{t} \right)^{1/2} \right\|_{L^p} \\
& \lesssim \| f\|_{L^p} \sup_{t\geq 1} \|\tphi(tL) F'(f) \|_{L^\infty},
\end{align*}
where we used the boundedness of the square functional. Then we conclude since $\tphi(tL) F'(f) $ is uniformly bounded by $\|F'(f)\|_{L^\infty}$ which is controlled by $\|f\|_{W^{s,p}}$ (due to Sobolev embedding with $s>d/p$ and Proposition \ref{prop:compoSobolev}).

Indeed our problem is to gain some extra regularity (from $s$ to $2s-d/p$) so the main difficulty relies on the study of the ``high frequencies'' and not on the lower ones. 

\mb {\bf Step 3:} Term $II$. \\

By duality and previous arguments, we get
$$ \|II\|_{W^{2s-d/p,p}} \lesssim \left\| \left( \int_0^1 t^{-2s+d/p}  |t^{1/2} X \tphi(tL) f|^4  \frac{dt}{t} \right)^ {1/2} \right\|_{L^p}+\left\| \left( \int_0^1  |t^{1/2} X \tphi(tL) f|^4  \frac{dt}{t} \right)^{1/2} \right\|_{L^p}$$
where we decomposed the norm with its homogeneous and its inhomogeneous parts and then used uniform boundedness of $F''(\tphi(tL) f)$.
Since (using $L^p$-boundedness of the Riesz transforms, see Assumption \ref{ass:riesz} and Sobolev embedding)
\begin{align*}
 \| X \tphi(tL) f| \|_{L^\infty} & \lesssim \| X \tphi(tL) f\|_{W^{d/p+\epsilon,p}} \\
  &  \lesssim  \| L ^{1/2} \tphi(tL) f\|_{L^p} + \| L^{1/2+d/{2p}+\epsilon/2} \tphi(tL) f\|_{L^p} \\
  & \lesssim  t^{s/2-1/2} \|f\|_{W^{s,p}} + t^{s/2-d/{2p}-\epsilon/2-1/2} \|f\|_{W^{s,p}} \\
  & \lesssim  t^{s/2-d/{2p}-\epsilon/2-1/2} \|f\|_{W^{s,p}}
\end{align*}
where we used $t< 1$.
Finally it comes, 
\begin{align*}
 \|II\|_{W^{2s-d/p,p}} & \lesssim \left\| \left( \int_0^1 t^{-s-\epsilon}    |t^{1/2} X \tphi(tL) f|^2  \frac{dt}{t} \right)^ {1/2} \right\|_{L^p}
& \lesssim \left\| \left( \int_0^1 |(t L)^{1/2-(s+\epsilon)/2} \tphi(tL) L^{(s+\epsilon)/2}f|^2  \frac{dt}{t} \right)^ {1/2} \right\|_{L^p} \\
& \lesssim\|f\|_{W^{s+\epsilon,p}},
 \end{align*}
where we used $s<1$ and the boundedness of the square functional.

\mb {\bf Step 4:} Terms $III$ and $IV$. \\
For these terms, we follow the reasoning of the Appendix of \cite{cm}.
Using the finite increments Theorem, we have
$$ \left| \tphi(tL) F'(f) - F'(\tphi(tL) f) \right| \leq \left| (\tphi(tL) -I)F'(f)\right|+ \left|(\tphi(tL)-I)f \right|.$$
So using similar arguments as previously, we get (with $h=F'(f)$ and $h=f$)
\begin{align*}
\| III \|_{W^{2s-d/p,p}} & \lesssim \left\| \left(\int_0^1 \left| \tphi(tL) F'(f) - F'(\tphi(tL) f))\right|^2\left| tL \tphi(tL) f \right|^2 t^{-s+d/(2p)} \frac{dt}{t}\right)^{1/2} \right\|_{L^p} \\
& \lesssim \left\| \left(\int_0^1 \left| (\tphi(tL) -I)h\right|^2\left| tL \tphi(tL) f \right|^2 t^{-2s+d/p} \frac{dt}{t}\right)^{1/2} \right\|_{L^p} \\
& \lesssim \left\| \left(\int_0^1 \left| (\tphi(tL) -I)h\right|^2 t^{-s} \frac{dt}{t}\right)^{1/2} \right\|_{L^p} \\
& \lesssim \left\| \left(\int_0^1 \left| \frac{\tphi(tL) -I)}{(tL)^{s} } L^{s} h\right|^2  \frac{dt}{t}\right)^{1/2} \right\|_{L^p} \\
& \lesssim \left\| L^{s} h\right\|_{L^p} \lesssim \|h\|_{W^{s,p}}
\end{align*}
where we used $s<2$, which yields
\be{eq:aaa} \left\| tL \tphi(tL) f \right\|_{L^\infty} \leq \left\| (tL)^{1-s/2} \tphi(tL) (tL)^{s/2}f \right\|_{L^\infty} \lesssim 
t^{-d/(2p)} \left\|(tL)^{s/2}f \right\|_{L^p} \lesssim t^{-d/(2p)+s/2} \|f\|_{W^{s,p}} \ee
and the boundedness of the square functional associated to the function $\frac{\tphi(z) -1}{z^{s}}$ which is holomorphic and vanishing at $0$ and at $\infty$ (see Proposition \ref{prop:Lp}).
We conclude the estimate of $III$ since $h=f$ or $h=F'(f)$ belongs to $W^{s,p}$.
The term $IV$ is similarly estimated.
\end{proof}

\begin{cor}
The diagonal term $\texttt{Rest}$ (defined in Corollary \ref{cor:decomp}) is bounded from $L^p \times L^q$ into $L^r$
as soon as $p,q\in(1,\infty]$ with $0<\frac{1}{r}:=\frac{1}{p}+\frac{1}{q}$. Moreover for $p\in(1,\infty)$, $\epsilon>0$ as small as we want and $s>d/p$ (with $s<(N-2)/2$), then 
$$ \left\| \texttt{Rest}(f,g) \right\|_{W^{2s-d/p, p}} \lesssim \|f\|_{W^{s+\epsilon,p}} \|g\|_{W^{s+\epsilon,p}}.$$
\end{cor}

\begin{proof} Apply Theorem \ref{thm:linearization} to the quantities $f+g$ and $f-g$ with $F(u):=u^2$. Then the polarization formulas give that $\texttt{Rest}(f,g)$ has the same regularity has $w$ in Theorem \ref{thm:linearization}.
\end{proof}

\begin{rem} Usually, we have a gain of regularity of order $s-d/p$ for this quantity. Here we have a gain of $s-d/p-\epsilon$ for every $\epsilon>0$, as smal as we want. 
\end{rem}

\begin{rem} \label{rem:imp} Of course, the assumption $s\in (d/p,1)$ can be seen as very constraining. We want to explain here how it seems to us possible to weaken that point. First we point that this ``technical'' difficulty is new since it does not appear in the Euclidean situation.

Legitimating the definition of the paraproducts (just before Definition \ref{def:paraproduit} and Remark \ref{rem:imp2}), we have developed  $(tL)^{1}(\tphi(tL) f \tphi(tL) F'(f))$ using the Leibniz rule of the Laplacian. Now for $M<< N$, it is possible to do the same operation and develop $(tL)^{M}(\tphi(tL) f \tphi(tL) F'(f))$. We also obtain a sum of different terms involving higher order differential operators than previously.
Similarly, in the beginning of the previous proof (obtaining the several terms $I,II,III,IV$ and $V$), we have developed $(tL)^1 F(\tphi(tL) f)$. Taking the same exponent $M$, we can expand $(tL)^M F(\tphi(tL) f)$ in different quantities. Indeed, we find different multilinear differential operators $T_{j,t}$ such that
\begin{align*}
 (tL)^M F(\tphi(tL) f) = & \sum_j T_{j,t}(\tphi(tL) f,F'(\tphi (tL)f), ..., F^{(2M-1)}(\tphi (tL)f)) \\
 & + F^{(2M)}(\tphi(tL) f) |t^{1/2} X \tphi(tL) f|^{2M}.
\end{align*}

By this way, we may define a new kind of paralinearization and prove that for $f\in W^{s+\epsilon,p}$ with large $s>d/p$ then  $$ F(f) = \sum_j \int_0^\infty T_{j,t}(\tphi(tL) f, \tphi (tL)F'(f), ..., \tphi (tL) F^{(2M-1)}(f)) \frac{dt}{t} + w $$
with $w\in W^{2s-d/p,p}$.
By this way, the rest $w$ should be decomposed as previously, in the corresponding term $II$, it will appear $|t^{1/2} X \tphi(tL) f|^{M-1}$ such that the corresponding square functional will be bounded as soon as $(s-d/p)(M-1)>s-1$.
The other terms may be bounded (as we did in Step 4) since they will have quantities as $|\tphi (tL)F^{(k)}(f) - F{(k)}(\tphi(tL) f)|$ and other differential operators on $\tphi(tL) f$. The key idea is that now the multilinearity of the operator $T_{j,t}$ will be sufficiently high to involve sufficiently such differential terms, each of them bringing a positive power of $t$ as shown in (\ref{eq:aaa}).

By this way, it is also possible to get a paralinearization result for high regularity $s>d/p$ and $s<<N$ (by taking a large exponent $N$), by defining new multilinear operators involving the derivatives $F^{(k)}(f)$.
\end{rem}

As explained in \cite{cm} (see its Appendix I.3, theorem 38), a vector-valued version of the preceding result allows us to prove the following one:
\begin{thm} \label{thm:linearization2} Consider $s\in(d/p,1)$, $f\in W^{s+k,p}$ and a smooth function $F(x,u_1,\cdots u_N) \in C^\infty(M\times \R^N)$ with $F(x,0, \cdots, 0)=0$. Then by identifying $\{1,\cdots, N\}$ with a set of multi-indices $\{\alpha_1,\cdots \alpha_N\}$ (and $|\alpha_i|\leq k$), we can build
\be{eq:Fbis} x \in M \rightarrow F(x, X_{\alpha_1} f(x), \cdots X_{\alpha_N}f(x)) \ee
which belongs to $W^{s,p}$. Moreover, 
\begin{equation} 
F(x,X_{\alpha_1}f(x),\cdots,X_{\alpha_N} f(x) ) = \sum_{i=1}^N \Pi_{[\partial_{u_i} F] (x,X_{\alpha_1}f(x),\cdots,X_{\alpha_N} f(x)) } ( X_{\alpha_i} f )(x) + w(x)  \label{eq:decompositbis} \end{equation}
with $w\in W^{2s-d/p,p}$. 
\end{thm}

\section{Propagation of low regularity for solutions of nonlinear PDEs}

As in the Euclidean case, paralinearization is a powerful tool to study nonlinear PDEs and to prove the propagation of regularity for solutions of such PDEs. Let us try to present some results in this direction with this new setting of Riemannian manifold.

Let us consider a specific case of nonlinear PDEs for simplifying the exposition : let $F(x,u_1,\cdots u_{\K+1}) \in C^\infty(M\times \R^{\K+1})$ be a smooth function with $F(x,0, \cdots, 0)=0$. Then by identifying $\{1,\cdots, \K+1\}$ with a set of multi-indices $\{0,1,\cdots \K\}$, we deal with the function 
\be{eq:F} F(f,Xf):=x \in M \rightarrow F(x, f(x),X_1f(x),  \cdots X_{\K}f(x)) \ee
for some function $f$. That corresponds to the case $N=\K+1$, $k=1$ with $\alpha_1=0$ and $\alpha_{i}=X_{i-1}$ for $i=2,\cdots N+1$ in (\ref{eq:F}).

\begin{thm} Consider $s\in(d/p,1)$, $f\in W^{s+1,p}$ and a smooth function (as above) $F(x,u_1,\cdots u_N) \in C^\infty(M\times \R^N)$ with $F(x,0, \cdots, 0)=0$ and assume that $f$ is a solution of
$$ F(f,Xf)(x)=0.$$
Consider the vector field
$$ \Gamma(x) := \sum_{i=2}^{\K+1} [\partial_{u_i} F] (x,f(x), X_{1}f(x),\cdots,X_{\K} f(x))  X_{i}.$$
Then, locally around each point $x_0\in M$ in ``the direction $\Gamma$'', the solution $f$ has a regularity $W^{s+1+\rho}$
for every $\rho>0$ such that
$$ \rho<\min\{1,s-d/p\}.$$
 In the sense that 
 $$ U(f):=\sum_{i=2}^{\K+1} [\partial_{u_i} F] (x,f(x), X_{1}f(x),\cdots,X_{\K} f(x))  L^{(s+\rho)/2} X_{i} (f) \in L^p.$$
 \end{thm}

Such results can be seen as a kind of directional ``Implicit function theorem'', where the regularity of $F(f,Xf)$ implies some directional regularity for $f$ (in the suitable direction, where we can regularly ``invert'' the nonlinear equation). 

\begin{proof}
The previous paralinearization result yields that
$$ \sum_{i=1}^\K \Pi_{[\partial_{u_{i+1}} F] (f,Xf)}(X_i(f)) \in W^{s+\rho,p},$$
which gives 
$$ T_F(f):=\sum_{i=1}^\K \widetilde{\Pi}_{[\partial_{u_{i+1}} F] (f,Xf)}(L^{\alpha} X_i f) \in L^p,$$
where $\widetilde{\Pi}$ is another paraproduct.
Indeed
\begin{align*}
 \widetilde{\Pi}_{b}(a) & =   - \int_{0}^\infty (tL)^\alpha \tilde{\psi}(tL) \left[(tL)^{1-\alpha} \tphi(tL)a \, \tphi(tL)b \right] \frac{dt}{t} \\
  & - \int_{0}^\infty (tL)^\alpha \tphi(tL) \left[ t^{-\alpha} \psi(tL)a \, \tphi(tL)b \right] \frac{dt}{t}, 
\end{align*}
where we have taken the notations of the definition for the initial paraproduct $\Pi$ (see Definition \ref{def:paraproduit}).
Then, we want to compare this quantity to the main one : $U(f) $. So let us examine the difference. Since for every constant $c$, we have
$$ cf = \Pi_{c}(f) = L^\alpha \Pi_c(L^{-\alpha} f) = \widetilde{\Pi}_{c}(f),$$
it comes 
$$U(f)(x) = \sum_{i=1}^\K \widetilde{\Pi}_{[\partial_{u_{i+1}} F] (f(x),Xf(x))}(L^\alpha X_i f)(x),$$
hence
$$ T_F(f)(x) - U(f)(x) = \sum_{i=1}^\K \widetilde{\Pi}_{\lambda_{i,x}}(X_i L^\alpha f)(x)$$
with $\lambda_{i,x}(\cdot)=[\partial_{u_{i+1}} F] (f,Xf)-[\partial_{u_{i+1}} F] (f(x),Xf(x))$.
It remains us to check that for each integer $i$, the function $x\rightarrow \widetilde{\Pi}_{\lambda_{i,x}}(X_i (1+L)^\alpha f)(x)$ belongs to $L^p$.
Let us recall that
\begin{align*}
 \widetilde{\Pi}_{\lambda_{i,x}}(X_i (1+L)^\alpha f)(x) & = - \int_{0}^\infty (tL)^\alpha \tilde{\psi}(tL) \left[(tL)^{1-\alpha} \tphi(tL) X_i L^\alpha f \, \tphi(tL)\lambda_{i,x}  \right] (x) \frac{dt}{t} \\
  & - \int_{0}^\infty (tL)^\alpha \tphi(tL) \left[ t^{-\alpha} \psi(tL) X_i L^\alpha f \, \tphi(tL)\lambda_{i,x} \right](x) \frac{dt}{t}.
 \end{align*}
Let us study only the first term $I$ (the second one beeing similar):
\begin{align*}
 I & := \left|\int_{0}^\infty (tL)^\alpha \tilde{\psi}(tL) \left[ (tL)^{1-\alpha} \tphi(tL) X_i L^\alpha f \, \tphi(tL)\lambda_{i,x}  \right] (x) \frac{dt}{t}\right| \\
  & \lesssim \int_{0}^\infty \int_M \frac{1}{\mu(B(x,t^{-1/2}))}\left(1+ \frac{d(x,y)}{t^{-1/2}}\right)^{-d-\delta} \left|(tL)^{1-\alpha} \tphi(tL) X_i L^\alpha f (y) \right| \left| \tphi(tL)\lambda_{i,x}(y) \right| \frac{d\mu(y) dt}{t} \\
  & \lesssim \int_{0}^\infty  \sum_{j\geq 0} 2^{-j\delta} \aver{C(x,2^jt^{-1/2})} \left|(tL)^{1-\alpha} \tphi(tL) X_i L^\alpha f (y) \right| \left| \tphi(tL)\lambda_{i,x}(y) \right| \frac{d\mu(y) dt}{t}.
\end{align*}
where we set $C(x,2^jt^{-1/2}) = B(x,2^{j+1}t^{-1/2})\setminus B(x,2^jt^{-1/2})$ and by convention $|C(x,2^jt^{-1/2})| = |C(x,2^jt^{-1/2})|$. 
Now, for $y\in B(x,2^j t^{-1/2})$, we have
\begin{align*}
 \left|\tphi(tL) \lambda_{i,x}(y)\right| & \lesssim \frac{1}{\mu(B(y,t^{1/2}))} \int \left(1+\frac{d(y,z)}{t^{1/2}}\right)^{-d-\delta} \left|[\partial_{u_{i+1}} F] (f,Xf)(z) -[\partial_{u_{i+1}} F] (f,Xf)(x) \right| d\mu(z) \\
 & \lesssim  \sum_{k\geq 0} 2^{jd-\delta k} \aver{C(y,2^{k+j} t^{1/2})} \left| H(z) -H(x) \right| d\mu(z)\\
 & \lesssim  \sum_{k\geq 0} 2^{jd-\delta k} \aver{\tilde{C}(x,2^{k+j} t^{1/2})} \left| H(z) -H(x) \right| d\mu(z)\\
 \end{align*}
with $H:=[\partial_{u_{i+1}} F] (f,Xf)$ and $\tilde{C}$ another systems of coronas.
So we get
\begin{align*}
I & \lesssim \sum_{k,j\geq 0} 2^{-k\delta +j(d-\delta)} \int_{0}^\infty \left(\aver{B(x,2^jt^{-1/2})} \left|(tL)^{1-\alpha} \tphi(tL) X_i L^\alpha f (y) \right| d\mu(y) \right) \\
& \hspace{2cm} \left( \aver{B(x,2^{k+j} t^{1/2})} \left| H(z) -H(x) \right| d\mu(z) \right) \frac{dt}{t} \\
& \lesssim \sum_{k,j\geq 0} 2^{-k\delta +j(d-\delta)}  \int_{0}^\infty {\mathcal M}\left[ t^{1/2} \left| (tL)^{1-\alpha} \tphi(tL) X_i L^\alpha f\right| \right](x)  \left( t^{-1/2}\aver{B(x,2^{k+j} t^{1/2})} \left| H(z) -H(x) \right| d\mu(z) \right) \frac{dt}{t}.
\end{align*}
Using Cauchy-Schwartz inequality, we also have
\begin{align*}
I & \lesssim \sum_{k,j\geq 0} 2^{-k\delta+j(d-\delta)} \left(\int_{0}^\infty \M\left[t^{1/2} (tL)^{1-\alpha} \tphi(tL) X_i L^\alpha f\right](x)^2 \frac{dt}{t}\right)^{1/2}  \\
& \hspace{2cm}  \left( \int_0^\infty t^{-1} \left(\aver{B(x,2^{k+j} t^{1/2})} \left| H(z) -H(x) \right| d\mu(z) \right)^ 2 \frac{dt}{t}\right)^{1/2} \\
& \lesssim \sum_{k,j\geq 0} 2^{-k(\delta-1)-j(d-\delta-1)} \left(\int_{0}^\infty \M\left[t^{1/2}(tL)^{1-\alpha} \tphi(tL) X_i L^\alpha f\right](x)^2 \frac{dt}{t}\right)^{1/2} \\
& \hspace{2cm} \left( \int_0^\infty t^{-1} \left(\aver{B(x,t^{1/2})} \left| H(z) -H(x) \right| d\mu(z) \right)^ 2 \frac{dt}{t}\right)^{1/2} \\
& \lesssim \left(\int_{0}^\infty \M\left[t^{1/2}(tL)^{1-\alpha} \tphi(tL) X_i (1+L)^\alpha f\right](x)^2 \frac{dt}{t}\right)^{1/2} \\
& \hspace{2cm} \left( \int_0^\infty t^{-1} \left(\aver{B(x,t^{1/2})} \left| H(z) -H(x) \right| d\mu(z) \right)^ 2 \frac{dt}{t}\right)^{1/2},
\end{align*}
where we have used a change of variables and $\delta>d+1$.
So using  exponents $q,r>p$ (later chosen) such that $\frac{1}{p}= \frac{1}{q}+\frac{1}{r}$, boundedness of the square functional on the one hand and on the other hand Fefferman-Stein inequality for the maximal operator, it comes
$$ \|I\|_{L^p} \lesssim  \|f\|_{W^{2\alpha,q}} \left\| \left( \int_0^\infty t^{-1} \left(\aver{B(x,t^{1/2})} \left| H(z) -H(x) \right| d\mu(z) \right)^ 2 \frac{dt}{t}\right)^{1/2} \right\|_{L^r}.$$
Then, using the characterization of Sobolev norms (using this functional, see Proposition \ref{prop:cara}), we conclude to
$$ \|I\|_{L^p} \lesssim  \|f\|_{W^{2\alpha,q}} \|H\|_{W^{1,r}}.$$
We also chose exponents $q,r$ such that
$$ \frac{1}{q}-\frac{2\alpha}{d}>\frac{1}{p}-\frac{s+1}{d} \quad \textrm{and} \quad \frac{1}{r}-\frac{1}{d}>\frac{1}{p}-\frac{s}{d},$$
which is possible since $\frac{1}{p}<\frac{s-\rho}{d}$ because of the condition on $\rho$.
Then Sobolev embedding (Proposition \ref{prop:Sobolev}) yields that $W^{s+1,p} \hookrightarrow W^{2\alpha,q}$ and 
$ W^{s,p} \hookrightarrow W^{1,r}$.
Finally, the proof is also concluded since we obtain
$$ \|I \|_{L^p} \lesssim \|f\|_{W^{s+1,p}} \|H\|_{W^{s,p}},$$
which is bounded by $f\in W^{s+1,p}$ (due to $H:=[\partial_{u_{i+1}} F] (f,Xf)$ with Proposition \ref{prop:compoSobolev}).
\end{proof}

We let the reader to write the analog results for higher order nonlinear PDEs.

\begin{rem} \label{rem:ass} Let us suppose that the geometry of the manifold allows us to use the following property:
For  $\alpha>0$, the commutators $[X_i,(1+L)^\alpha]$ is an operator of order $2\alpha$, which means that for all $p\in(1,\infty)$ and $s>0$, $[X_i,(1+L)^\alpha]$ is bounded from $W^{s+2\alpha,p}$ to $W^{s,p}$.

This property holds as soon as we can define a suitable pseudo-differential calculus with symbolic rules : in particular, this is the case  of H-type Lie groups, using a notion of Fourier transforms based on  irreductible representations, see \cite{BFG,GS}.

Under this property, we can commute the vector field $X$ with any power of the Laplacian and so with the same statement than in the previous theorem, we obtain that 
$$ \Gamma (1-L)^{\frac{s+\rho}{2}} f \in L^p.$$
This new formulation better describes the fact that $f\in W^{s+\rho+1,p}$ along the vector field $\Gamma$.
 \end{rem}

\end{document}